\documentclass{amsart}
\pdfoutput=1

\usepackage{graphicx}
\usepackage{hyperref}

\usepackage[all]{xy}

\usepackage{subfigure}   

\renewcommand{\labelenumi}{(\arabic{enumi})}

\newcommand{\enumia}{\renewcommand{\labelenumi}{(\arabic{enumi})}}
\newcommand{\enumir}{\renewcommand{\labelenumi}{(\roman{enumi})}}

\newtheorem{theorem}{Theorem}

\newtheorem{cor}[theorem]{Corollary}
\newtheorem{lem}[theorem]{Lemma}
\newtheorem{prop}[theorem]{Proposition}

\newtheorem*{tilingconjecture}{Tiling Conjecture}
\newtheorem*{vtilingconjecture}{Virtual Tiling Conjecture}
\newtheorem*{theoremnonumber}{Theorem}
\newtheorem*{corollarynonumber}{Corollary}

\theoremstyle{definition}
\newtheorem{setting}[theorem]{Setting}
\newtheorem{definition}[theorem]{Definition}
\newtheorem{que}[theorem]{Question}
\newtheorem{rem}[theorem]{Remark}
\newtheorem{exmp}[theorem]{Example}

\def\co{\colon\thinspace}
\def\cobar{|\thinspace}

\newcommand{\cay}{\ensuremath{\mathrm{Cay}}}
\newcommand{\underw}{\ensuremath{\underline{w}}}
\newcommand{\hatw}{\ensuremath{\hat{w}}}

\newcommand{\bbb}[1]{\ensuremath{\mathbb{#1}}}
\newcommand{\R}{\bbb{R}}
\newcommand{\Z}{\bbb{Z}}

\newcommand{\script}[1]{\ensuremath{\mathcal{#1}}}
\def\DD{\script{D}}
\def\SS{{\script S}}

\newcommand{\smallcaps}[1]{\textrm{\textsc{#1}}}
\newcommand{\aut}{\smallcaps{Aut}}

\newcommand{\link}{\smallcaps{Link}}

\begin{document}
\title{Geometricity and Polygonality in Free Groups}
\author{Sang-hyun Kim}
\address{Department of Mathematics, the University of Texas at Austin}
\email{shkim@math.utexas.edu}

\begin{abstract}    
Gordon and Wilton recently proved that the double $D$ of a free group $F$ amalgamated along a cyclic subgroup $C$ of $F$ contains a surface group if a generator $w$ of $C$ satisfies a certain $3$--manifold theoretic condition, called virtually geometricity. Wilton and the author defined the polygonality of $w$ which also guarantees the existence of a surface group in $D$. In this paper, virtual geometricity is shown to imply polygonality up to descending to a finite-index subgroup of $F$. That the converse does not hold will follow from an example formerly considered by Manning.
\end{abstract}

\maketitle
\date{\today}

\section{Introduction}\label{sec:introduction}
A \textit{surface group} (\textit{hyperbolic surface group}, respectively) 
means a group isomorphic to the fundamental group of a closed surface with non-positive (negative, respectively) Euler characteristic. Throughout this paper, we let $F$ be a finitely generated non-abelian free group, and $H$ be a handlebody so that $\pi_1(H)=F$.
Unless specified otherwise, we fix  $\SS=\{a_1,\ldots,a_n\}$ as a generating set for $F$.
Each element $w\in F$ can be written as a \textit{word} written in $\SS\cup\SS^{-1}$.
That means, $w=v_1 v_2\cdots v_l$ for some $v_i\in \SS\cup\SS^{-1}$. Each $v_i$ is called a \textit{letter} of $w$. $w$ is \textit{cyclically reduced} if $v_i\ne v_{i+1}^{-1}$, where the indices are taken modulo $l$. The length of a cyclically reduced word $w$ is notated as $|w|$.
The Cayley graph of $F$ with respect to the generating set $\SS$ is an infinite $2n$--valent tree, denoted as $\cay_\SS(F)$ or simply $\cay(F)$. There is a natural action of $F$ on $\cay(F)$, so that $\cay(F)/F$ is a bouquet of $n$ circles. Each circle in $\cay(F)/F$ inherits from $\cay(F)$ an orientation and a label by an element in $\SS$.

Motivated by $3$--manifold theory, Gromov conjectured every one-ended word-hyperbolic group contains a surface group~\cite{Bestvina:2009p3862}. Note that a surface subgroup of a word-hyperbolic group is actually a hyperbolic surface group.
We say that $U\subseteq F$ is \textit{diskbusting} if there does not exist a non-trivial free decomposition $F=G_1\ast G_2$ such that each element of $U$ is conjugate into one of $G_i$'s \cite{Canary:1993p4322,Stong:1997p4326,Stallings:1999p3173}. For convention, we will only consider finite subsets of $F$. As a special case, a word $w\in F$ is diskbusting if $w$ does not belong to any proper free factor of $F$. $w\in F$ is \textit{root-free} if $w$ is not a proper power. A particularly simple case of the Gromov's conjecture is when a word-hyperbolic group is given as the \textit{double} $D(w)= F\ast_{\langle w\rangle} F$. $D(w)$ is one-ended and word-hyperbolic if and only if $w$ is diskbusting and root-free~\cite{Bestvina:1992p456,Gordon:2009p360}. Gordon and Wilton ignited attention on this case by providing several sufficient conditions for $D(w)$ to contain a surface group~\cite{Gordon:2009p360}. On one hand, they formulated a homological condition by crucially using a result of Calegari on surface subgroups of word-hyperbolic graphs of free groups with cyclic edge groups~\cite{Calegari:2008p1810}. On the other hand, they considered a $3$--manifold theoretic condition as follows. Realize $U\subseteq F=\pi_1(H)$ as an embedded $1$--submanifold $A\subseteq H$. $U$ (or equivalently, $A$) is said to be \textit{virtually geometric} if there exists a finite cover $p:H'\rightarrow H$ such that $p^{-1}(A)$ is freely homotopic to a $1$--submanifold embedded in $\partial H'$. In particular, $U$ (or equivalently, $A$) is \textit{geometric} if $A$ is freely homotopic to a $1$--submanifold on $\partial H$. 
If $w\in F$ is virtually geometric and diskbusting, then $D(w)$ contains a surface group~\cite{Gordon:2009p360}. However, not all root-free diskbusting words are virtually geometric~\cite{Manning:2009p3177}.

Let $w\in F = \pi_1(\cay(F)/F)$ be cyclically reduced.
Following \cite{Stallings:1983p596}, a locally injective graph map
is called an \textit{immersion}.
There exists an immersed loop $\gamma\subseteq \cay(F)/F$ such that $[\gamma]=w$;
in this case, we say that $\gamma$ \textit{reads} $w$.
Let $X(w)$ denote the 2--dimensional CW-complex obtained by taking two copies of $\cay(F)/F$ 
and gluing each boundary component of a cylinder along each copy of $\gamma_w\subseteq \cay(F)/F$.
In the case when the free basis $\SS$ needs to be explicitly notated, we write
$X_\SS(w)=X(w)$.
It turns out that $X(w)$ is an Eilenberg--Mac Lane space for $D(w)$~\cite{Wise:2000p790,Gordon:2009p360}.

In~\cite{Kim:2009p3867},
Wilton and the author defined a combinatorial condition for a cyclically reduced word $w$, 
called \textit{polygonality}; see Definition~\ref{defn:polygonal}.
A key observation was, the polygonality of $w$ is equivalent to the existence of a homeomorphically and $\pi_1$--injectively embedded closed surface of non-positive Euler characteristic in some finite cover of $X(w)$.
Hence, the polygonality of $w$ will guarantee the existence of a surface group in $D(w)$.

We will use the following terminology that appears in~\cite{Kim:2009p3867}.
A \textit{polygonal disk} $P$ is a $2$--dimensional closed disk with a CW--structure such that $\partial P = P^{(1)}$ is a polygon.
Let $E(\partial P)$ denote the set of the edges in $\partial P$.
A \textit{side-pairing} $\sim$ on a collection of polygonal disks $P_1,\ldots,P_m$ is 
a partition of $\coprod_i E(\partial P_i)$ into unordered pairs,
along with a choice of a homeomorphism between the edges in each pair;
here, we require that such a homeomorphism does not identify two consecutive edges of any $\partial P_i$ 
in a way that fixes a common vertex of the two edges. 
A side-pairing $\sim$ on polygonal disks $P_1,\ldots,P_m$ determines an identification of the edges of $\coprod_i P_i$; 
hence, we have a closed surface $S=\coprod_i P_i/\!\!\sim$. 
We denote by $m(S)$ the number of polygonal disks ($2$--cells) which $S$ is made of.
If $\phi\co\Gamma\rightarrow\cay(F)/F$ is a graph map,
each $e\in E(\Gamma)$ carries an orientation and a label (by $\SS$)
induced from the orientation and the label of $\phi(e)\in E(\cay(F)/F)$.
Conversely, for a graph $\Gamma$, 
a choice of an orientation and a label (by $\SS$) for each $e\in E(\Gamma)$ 
determines a graph map $\phi\co\Gamma\rightarrow\cay(F)/F$;
if $e\in E(\Gamma)$ is labeled by $a_i$, we call $e$ as an $a_i$-edge of $\Gamma$.
That $\phi\co\Gamma\rightarrow\cay(F)/F$ is an immersion is equivalent to that
for each $a_i\in\SS$ and each $v\in\Gamma^{(0)}$,
there do not exist two incoming $a_i$-edges or two outgoing $a_i$-edges at $v$.
We say that $U\subseteq F$ is \textit{independent} if for any two distinct $w_1,w_2\in U$ and for any $t_1,t_2\in\Z\setminus\{0\}$,
$w_1^{t_1}$ is not conjugate to $w_2^{t_2}$.
The definition of polygonality generalized to an independent set of cyclically reduced words is as follows;
see \cite{Kim:2009p3867} for the case when $|U|=1$.

\begin{definition}\label{defn:polygonal}
Let $U\subseteq F$ be an independent set of cyclically reduced words.
\begin{enumerate}
\item
Suppose there exists a side-pairing $\sim$ on polygonal disks $P_1,\ldots,P_m$
and an immersion $\phi\co S^{(1)}\rightarrow\cay(F)/F$ where $S=\coprod_i P_i/\!\!\sim$,
such that the composition $\partial P_i\rightarrow S^{(1)}\rightarrow\cay(F)/F$
reads a non-trivial power of some word in $U$.
Then the closed surface $S$ is called a \textit{$U$--polygonal surface}.
\item
$U$ is \textit{polygonal} 
if either 
$U$ contains a proper power or there exists a
$U$--polygonal surface $S$ such that $\chi(S) < m(S)$.
\end{enumerate}
\end{definition}

\begin{rem}\label{rem:defn:polygonal}
Let $U\subseteq F$ be a set of cyclically reduced words. 
The condition $\chi(S)<m(S)$  is necessary for the polygonality to be a non-trivial property, for one can always construct a $U$--polygonal surface $S$ with $\chi(S)=m(S)$ as follows.
Choose any non-trivial word $w\in U$ and take two polygonal disks $P$ and $P'$ 
each of whose boundary components reads $w$.
For $e\in E(\partial P)$ and $e'\in E(\partial P')$, define $e\sim e'$ if and only if $e$ and $e'$ read the same letter of $w$.
Then for $S=P\coprod P'/\!\!\sim$, $\chi(S)=m(S)=2$. One can also construct a $U$--polygonal surface $S= P/\!\!\sim$ with $\chi(S)=m(S)=1$ by taking a polygonal disk $P$ whose boundary reads $w^2$ and by defining a side-pairing $\sim$ to identify  the edges of $\partial P$ with respect to the $\pi$--rotation.
\end{rem}

The polygonality of $U\subseteq F$ depends on the choice of a free basis $\SS$ for $F$ in which $U$ is written.
For $g,h\in F$, we let $g^h$ denote $h^{-1}g h$.
Choose  a finite-index subgroup  $F'$ of $F$ and a free basis  $\SS'$  for $F'$.
For $w'\in F'$, let us denote the $F'$--conjugacy class of $w'$ by  $[w']$.
Fix a word $w\in F$.
For $g\in F$,
let $n_g$ be the smallest positive integer such that $(w^{n_g})^g\in F'$.
Following~\cite{Gordon:2009p360,Manning:2009p3177},
define \[\underw_{F'} = \{ [(w^{n_g})^g] \cobar gF'\in F/F'\}.\]
Note that $|\underw_{F'}|\le |F/F'|$.
We define a \textit{transversal $\hatw_{F'}$ for $\underw_{F'}$} to be a subset of $F'$ obtained by
choosing exactly one element from each conjugacy class in $\underw_{F'}$;
here, we will regard $\hatw_{F'}$ as a set of cyclically reduced words written in $\SS'$,
by taking cyclic conjugations if necessary. 
Suppose for some $g$ and $h$ in $F$, $(w^{n_g})^g$ and $(w^{n_h})^h$ have non-trivial powers which are conjugate to each other in $F'$.
Since $F$ does not have any non-trivial Baumslag--Solitar relation,
we can write $(w^M)^g = (w^{M})^{hf'}$ for some $M>0$ and $f'\in F'$.
We have $hf'g^{-1}\in C(w^{M})=C(w)$ and so, $w^{hf'} = w^g$.
In particular, $n_g=n_h$ and $[(w^{n_g})^g] = [(w^{n_h})^h]$. This shows that $\hatw_{F'}$ is independent as a set of words in $F'$.

\begin{definition}\label{defn:vp}
A word $w\in F$ is \textit{virtually polygonal} if 
a transversal $\hatw_{F'}$ for $\underw_{F'}$ is polygonal as 
an independent set of cyclically reduced  words written in $\SS'$,
for some finite-index subgroup $F'$ of $F$ and for some free basis $\SS'$ of $F'$.
\end{definition}

After formulating polygonality in terms of Whitehead graphs,
we will prove a relation between virtual geometricity and virtual polygonality.

 
\begin{theoremnonumber}[Theorem~\ref{thm:vg implies vp}]
A diskbusting and virtually geometric word in $F$ is virtually polygonal.
\end{theoremnonumber}

While Theorem~\ref{thm:vg implies vp} is interesting in its own, it also has a corollary related to the Gromov's conjecture on $D(w)$. 


\begin{theoremnonumber}[Theorem~\ref{thm:vp}]
If a root-free word $w\in F$ is virtually polygonal, then $D(w)$ contains a hyperbolic surface group.
\end{theoremnonumber}

\begin{corollarynonumber}[Corollary~\ref{cor:vp}; first proved by Gordon--Wilton~\cite{Gordon:2009p360}]
	If $w\in F$ is root-free, virtually geometric and diskbusting, then $D(w)$ contains a hyperbolic surface group.
\end{corollarynonumber}

In~\cite{Manning:2009p3177}, Manning proved that the word $w_1=bbaaccabc\in F_3=\langle a,b,c\rangle$ is not virtually geometric. The same argument shows that $w_2 =aabbacbccadbdcdd \in F_4=\langle a,b,c,d \rangle$ is not virtually geometric. We will prove that $w_1$ and $w_2$ are both polygonal (Proposition~\ref{prop:w1w2}). As a consequence, $D(w_1)$ and $D(w_2)$ contain surface groups.

\begin{theoremnonumber}[Theorem \ref{thm:p but not vg}]
	There exist polygonal words which are not virtually geometric.
\end{theoremnonumber}

\section{Preliminary}\label{sec:preliminary}
We recall basic facts on polygonality and Whitehead graphs.
The material in Section~\ref{subsec:polygonality} is largely drawn from~\cite{Kim:2009p3867}.
For readers' convenience, we include a complete proof of Theorem~\ref{thm:polygonal}
which is similar to the argument in~\cite[Corollary 2.11]{Kim:2009p3867}.

\subsection{Polygonality}\label{subsec:polygonality}
We say that $U,U'\subseteq F$ are \textit{equivalent} if $U'=\phi(U)$ for some $\phi\in \aut(F)$.
The definition of polygonality of $U\subseteq F$ involves the representation of $U$ as a set of words written in $\SS=\{a_1,\ldots,a_n\}$~(Definition~\ref{defn:polygonal}). In particular, $U\subseteq F$ is equivalent to a polygonal set of words if and only if $U\subseteq F$ is polygonal with respect to some choice of a free basis for $F$. We list some of the known polygonal words as follows.
\begin{exmp}\label{exmp:polygonal}
(1)	Consider a cyclically reduced word $w\in F = \langle a_1,\ldots,a_n\rangle$ 
	such that for each $i$, exactly two letters of $w$ belong to $\{a_i,a_i^{-1}\}$.
	In particular, $|w|=2n$.
	We claim that $w$ is polygonal.
	We may assume $w$ is root-free, as any proper powers are polygonal by definition.
	Let $P$ be a polygonal disk
	with each edge oriented and labeled by $\SS$
	such that $\partial P$ reads $w$.
 	Let $\sim$ be the side-pairing which identifies the
	two $a_i$-edges of $\partial P$ for each $i$.
	Put $S=P/\!\!\sim$. Since $S^{(1)}$ has exactly one $a_i$-edge for each $i$,
	$S^{(1)}$ immerses into $\cay(F)/F$.
	In order to prove $w$ is polygonal, we have only to show that $\chi(S)<m(S)=1$.
	For each $v\in S^{(0)}$, 
	let $d_v$ denote the valence (of the $1$--complex $S^{(1)}$) at $v$.
	As $w$ is cyclically reduced, $d_v\ge2$ for each $v$. Moreover,
	$\chi(S)=|S^{(0)}|-|S^{(1)}|+1 = \sum_{v\in S^{(0)}}1 - \sum_{v\in S^{(0)}} d_v/2 + 1$.
	Hence, $\chi(S)<1$ fails only when $d_v=2$ for each $v\in S^{(0)}$.
	If $d_v=2$ for each $v$, then one can see that $S\approx\R P^2$ and $\sim$ identifies the edges of $\partial P$ by the $\pi$--rotation; this means, $\sim$ identifies $v_i$ to $v_{(|w|/2)+i}$ for each $i$. 
	It follows that $w=u^2$ for some $u\in F$, and we have a contradiction.

(2)
	Let $F_2$ denote the free group generated by $a$ and $b$.
	If $|p_i|,|q_i|>1$ for each $i$, then
	the words $\prod_i a^{p_i} b^{q_i}, 
	\prod_i a^{p_i} (a^b)^{q_i}\in F_2$ are polygonal~\cite{Kim:2009p3867}.

(3)
	The word $w= \prod_i a^{p_i} (a^b)^{q_i}$ is called a 
	\textit{positive height--1 word} of $F_2$ if $p_i,q_i>0$.
	Put $p=\sum_i p_i, q=\sum_i q_i, p'=\sum_{p_i=1} 1$ 
	and 
	$q'=\sum_{q_i=1} 1$.
	If $pp'\le q^2$ and $qq'\le p^2$, then $w$ is polygonal~\cite{Kim:2009p3867}.
	With a suitable notion of probability on $F_2$,
	this implies a positive height--1 word is ``almost surely'' polygonal.
\end{exmp}
	
Let $p:(X',x')\rightarrow (X,x)$ be a finite covering map for some based spaces $(X',x')$ and $(X,x)$.
For a based loop $\gamma\co (S^1,v)\rightarrow (X,x)$,
let $q\co (S^1,v')\rightarrow (S^1,v)$ be the smallest covering such that
$\gamma \circ q$ lifts to a based loop $\tilde{\gamma}\co:(S^1, v')\rightarrow (X',x')$ 
as shown in the commutative diagram below.
Following~\cite{Wise:2000p790},
we say that $\tilde{\gamma}$ is the \textit{elevation  of $\gamma$ at $x'$  with respect to $p$}.
\[
\xymatrix{
 & (S^1,v') \ar[rr]^{\tilde{\gamma}}\ar[d]^q && (X',x')\ar[d]^p\\
 & (S^1,v) \ar[rr]^{{\gamma}} && (X,x)\\
}
\]

Let $U=\{w_1,\ldots,w_r\}\subseteq F$ be an independent set of root-free, cyclically reduced words,
and $\gamma_i$ be the (based) loop in $\cay(F)/F$ reading $w_i$.
Take two copies $\Gamma_1,\Gamma_2$ of $\cay(F)/F$
and glue each boundary component of a cylinder $C_i\approx S^1\times[-1,1]$ along each copy of $\gamma_i$ in $\Gamma_1$ and in $\Gamma_2$, for $i=1,2,\ldots,r$.
We let $X(U)$
denote the
$2$--dimensional CW-complex thus obtained.
Define $D(U)=\pi_1(X(U))$.
Take $F^{(1)},F^{(2)}\cong F$, and let $w_i^{(j)}$ denote the image of $w_i$ in $F^{(j)}$ for $j=1,2$.
Note that 
\[D(U) \cong \langle F^{(1)},F^{(2)},t_2,t_3,\ldots,t_r \cobar w_1^{(1)}=w_1^{(2)},
(w_i^{(1)})^{t_i}=(w_i^{(2)})\mbox{ for }i=2,\ldots,r\rangle.\]
For a finite-index subgroup $F'\le F$, 
$\cay(F)/F'$ is the finite covering space of $\cay(F)/F$ corresponding to $F'$.
For each choice of the basepoint $\tilde{x}$ of $\cay(F)/F'$, 
there is an elevation of $\gamma_i$ at  $\tilde{x}$ with respect to the covering $\cay(F)/F'\rightarrow\cay(F)/F$.
We identify two elevations $\tilde{\gamma},\tilde{\gamma}'$ 
(at two distinct basepoints $\tilde{x},\tilde{x}'\in\cay(F)/F'$)
of the based loop $\gamma_i\co S^1\rightarrow\cay(F)/F$ 
if $\tilde{\gamma}$ and $\tilde{\gamma}'$ are the same 
as loops without basepoints (\textit{i.e.}, when the basepoints are forgotten); this means, the lifting of some power of $\gamma_i$ at $\tilde{x}$ to $\cay(F)/F'$ terminates at $\tilde{x}'$. See~\cite[Lemma 2.7]{Wilton:2008p7013} and~\cite[Lemma 2.4]{Kim:2009p3867} for algebraic description of this identification.
We let $\{ \tilde{\gamma}_1,\tilde{\gamma}_2,\ldots,\tilde{\gamma}_s\}$ denote the set of all the elevations of $\gamma_1,\gamma_2,\ldots,\gamma_r$ at the vertices of $\cay(F)/F'$, 
after this identification.
Take two copies $\Gamma_1',\Gamma_2'$ of $\cay(F)/F'$ and glue each boundary component of a cylinder $C_j'$ along each copy of $\tilde{\gamma}_j$ in $\Gamma_1'$ and in $\Gamma_2'$, for $j=1,2,\ldots,s$; in this way, one obtains a finite cover $Y(U,F')$ of $X(U)$.
The image of each cylinder in $X(U)$ or in $Y(U,F')$ will still be called a \textit{cylinder}.
Lemma~\ref{lem:cylinder} is a special case of \cite[Lemma 2.2]{Kim:2009p3867}. We omit the proof here.

\begin{lem}[\cite{Kim:2009p3867}]\label{lem:cylinder}
Suppose $U\subseteq F$ is an  independent set of root-free, cyclically reduced  words and $[F:F']<\infty$. Let $S$ be a closed connected surface homeomorphically embedded in
$Y(U,F')$. Then $S$ is the union of some cylinders, $\chi(S)\le0$ and $\pi_1(S)$ embeds into $D(U)$.\qed
\end{lem}

The following generalizes \cite[Corollary 2.11]{Kim:2009p3867}.

\begin{theorem}\label{thm:polygonal}
Let $U\subseteq F$ be an independent set of root-free, cyclically reduced  words.
If $U$ is polygonal, then $D(U)$ contains a hyperbolic surface group.
\end{theorem}

\begin{proof}
	Write $U = \{w_1,\ldots,w_r\}$, and let $\gamma_j\subseteq \cay(F)/F$ realize $w_j$.
	Let $\sim$ be a side-pairing on polygonal disks $P_1,\ldots,P_m$
	such that
	 $S=\coprod_i P_i/\!\!\sim$ is a closed $U$--polygonal surface satisfying	$\chi(S)<m$.
	This implies that for some immersion $\phi\co S^{(1)}\rightarrow\cay(F)/F$,
 	each composition $\partial P_i\rightarrow S^{(1)}\rightarrow\cay(F)/F$
	reads a non-trivial power of an element in $U$.
	Since $F$ is subgroup separable, 
	one can lift the immersion $\phi$ to an embedding
	$\phi'\co S^{(1)}\rightarrow \cay(F)/F'$
	for some $[F:F']<\infty$~\cite{Scott:1978p142}.
	Choose an open disk $B_i$ in the interior of each $P_i$ and 
	let $S'$ be the double of $S\setminus\cup_i B_i$.
	Since $\phi'$ maps each $\partial P_i$ to an elevation of some $\gamma_j$,
	the definition of $Y(U,F')$ implies that
	$S'$ is homeomorphic to the union $S''$ of some cylinders in $Y(U,F')$.
	By Lemma~\ref{lem:cylinder}, $\pi_1(S')$ embeds into $D(U)$.
	Note that $\chi(S')=2 (\chi(S)-m) <0$.
\end{proof}

\begin{rem}\label{rem:polygonal}
Let $U \subseteq F$ be as in the hypothesis of Theorem~\ref{thm:polygonal}.
By an elementary argument on graphs of spaces,
a finite cover of $X(U)$ contains a homeomorphically embedded closed surface if and only if
so does $Y(U,F')$ for some $[F:F']<\infty$.
Then Theorem~\ref{thm:polygonal} can actually be strengthened as follows:
$U$ is polygonal if and only if a finite cover of $X(U)$ contains a  homeomorphically embedded closed hyperbolic surface;
see~\cite{Kim:2009p3867}.
\end{rem}

\subsection{Whitehead graph}

A \textit{graph} means a $1$--dimensional CW-complex. 
For a graph $\Gamma$, let $V(\Gamma)$ and $E(\Gamma)$ denote the vertex set and the edge set, respectively.
For a cyclically reduced word $w = v_1 v_2\ldots v_l\in F$ where $v_i\in \SS\cup\SS^{-1}$,
we let $W(w)$ denote the \textit{Whitehead graph} of $w$. This means,
$W(w)$ is a graph with the vertex set $ \SS\cup\SS^{-1}$
and the edge set $\{e_1,e_2,\ldots,e_l\}$
such that $e_i$ joins $v_i$ and $v_{i+1}^{-1}$, where the indices are taken modulo $l$.
In the special case when $w=a_i$ or $w=a_i^{-1}$,
 $W(w)$ consists of a single edge joining $a_i$ and $a_i^{-1}$.
We define the \textit{connecting map $\sigma_w$ associated with $W(w)$} as the map
$\sigma_w\co\{(e,v)\cobar e\in E(W(w)), v\in \partial e\}\rightarrow E(W(w))$
such that
$\sigma_w(e_i,v_i)=e_{i-1}$ and $\sigma_w(e_i,v_{i+1}^{-1})=e_{i+1}$.
In particular, if $\sigma(e,v)=e'$ then $v\in\partial e$, $v^{-1}\in\partial e'$ and moreover,
the length--$2$ subwords of $w$ corresponding to $e$ and $e'$ are consecutive and share the letter $v$ or $v^{-1}$;
see Example~\ref{exmp:cnt}.

Consider a non-zero integral polynomial $f(x_1,\ldots,x_r)=\sum_{1\le i\le r,1\le j\le s} c_{ij} x_i^j$
where $s>0$ and $c_{ij}\ge0$.
For $U=\{w_1,\ldots,w_r\}\subseteq F$ consisting of cyclically reduced words,
we define \[W(f(U))=W(f(w_1,\ldots,w_r))=\cup_{i=1}^r \cup_{j=1}^s \cup_{k=1}^{c_{ij}} W(w_i^j)\]
where the union is taken so that each term $W(w_i^j)$ has the common vertex set $\SS\cup \SS^{-1}$,
and any two terms do not have a common edge.
For instance,
$W(w+w)$ is obtained by doubling all the edges of $W(w)$, so that
$V(W(w+w))=\SS\cup\SS^{-1}$ and $|E(W(w+w))|=2|E(W(w))| = 2|w|$.
Note that $W(f(U))$ is an abusive notation in that $U$ is considered as an \textit{ordered tuple}
rather than just a set.
We have 
$|E(W(f(U)))| = \sum_{i,j}c_{ij} |E(W(w_i^j))| = \sum_{i,j}c_{ij} j|w_i|$.
The unique extension on $W(f(U))$ of
the connecting map on each $W(w_i^j)$ is called
the \textit{connecting map associated with $W(f(U))$}.

\begin{exmp}\label{exmp:cnt}.
Consider the word $w=ab^{-2}\in F_2=\langle a,b\rangle$. Write $w=v_1v_2v_3$ where $v_1=a$ and $v_2=v_3=b^{-1}$.
$W(w)$ is shown in Fig.~\ref{fig:connecting map} (a);
here, the edges $e_1,e_2$ and $e_3$ correspond to the (cyclic) subwords $v_1v_2=ab^{-1}, v_2v_3=b^{-2}$ and $v_3v_1=b^{-1}a$, respectively. Some of the values of the associated connecting map are given as $\sigma_w(e_1,a)=\sigma_w(e_1,v_1)=e_3$, $\sigma_w(e_3,a^{-1})=\sigma_w(e_3,v_1^{-1})=e_1$ and $\sigma_w(e_3,b^{-1})=\sigma_w(e_3,v_3)=e_2$.

Figure~\ref{fig:connecting map} (b) shows $W(w+w)=W(2w)$. $e_i$ and $e_i'$ denote the edges corresponding to the length--$2$ cyclic subword $v_iv_{i+1}$ of $w$. The connecting map is given as $\sigma_{2w}(e_1,a)=e_3$, $\sigma_{2w}(e_3,a^{-1})=e_1$,$\sigma_{2w}(e_3,b^{-1})=e_2$, $\sigma_{2w}(e_1',a)=e_3'$ and so forth.

$W(w^2)$ is drawn in Fig.~\ref{fig:connecting map} (c). $f_i$ denotes the edge corresponding to the $i$--th length--$2$ cyclic subword of $w^2=ab^{-2}ab^{-2}$. The connecting map can be computed as $\sigma_{w^2}(f_1,a)=f_6$, $\sigma_{w^2}(f_3,a^{-1})=f_4$, $\sigma_{w^2}(f_3,b^{-1})=f_2$, $\sigma_{w^2}(f_4,a)=f_3$ and so forth. Note that $W(2w)$ and $W(w^2)$ are the same as graphs, while $\sigma_{2w}$ and $\sigma_{w^2}$ are distinct.

\end{exmp}

\begin{figure}[htb!]
	\subfigure[$W(w)$.]{
	\includegraphics[]{fig/figcnt.1}
	}\hfill
	\subfigure[$W(w+w)=W(2w)$.]{
	\includegraphics[]{fig/figcnt.2}
	}\hfill
	\subfigure[$W(w^2)$.]{
	\includegraphics[]{fig/figcnt.3} 
	}
	\caption{$w=ab^{-2}$.
	\label{fig:connecting map}}
\end{figure}

\begin{rem}\label{rem:connecting map}
(1)
Let $\sigma$ be the connecting map associated with a Whitehead graph $\Gamma$.
For $e\in E(\Gamma)$ and
$v\in\partial e$, $v^{-1}$ is an endpoint of $\sigma(e,v)$ and
$\sigma(\sigma(e,v),v^{-1})=e$.

(2)
Let $U=\{w_1,\ldots,w_r\}$ and $f(x_1,\ldots,x_r)=\sum_{i,j\ge1}c_{ij} x_i^j$.
For each $i$ and $j$, 
glue $c_{ij}$ copies of a disk to $\cay(F)/F$ along the loop reading $w_{i}^j$.
Let $Z(f(U))$ be the $2$--dimensional CW-complex thus obtained.
Then $W(f(U))$ is precisely the link of the (unique) vertex $v$ in $Z(f(U))$.
The vertex $a_i$ ($a_i^{-1}$, respectively) of $W(f(U))$ corresponds to the 
incoming (outgoing, respectively) portion of the $a_i$-edge in $Z(f(U))^{(1)}$ at $v$.
\end{rem}

A collection $\DD=\{D_1,\ldots,D_n\}$ of disjoint, properly embedded disks in $H$ is called
a \textit{disk structure on $H$} if $H\setminus\cup_i D_i$ is a $3$--cell~\cite{Stallings:1999p3173}.
We will alway equip each disk $D_i$ with a transverse orientation.
Let $A$ be an embedded $1$--submanifold of $H$. 
Assume that $A$ intersects $\DD$ \textit{transversely and minimally}.
This means, 
$A\cap(\cup_i D_i)$ consists of finitely many points
and $|A\cap (\cup_i D_i)|\le |A'\cap (\cup_i D_i)|$ 
for any $1$--submanifold $A'$ freely homotopic to $A$.
Given a disk structure
$\DD$ on $H$, we denote by $\DD\times[-1,1]=\cup_i (D_i\times[-1,1])$ 
a closed regular neighborhood of $\DD$ in $H$ so that $D_i$ is identified with $D_i\times0$
and the transverse orientation of $D_i$ is from $D_i\times{-1}$ to $D_i\times1$.
By transversality, we may assume that $A\cap (D_i\times[-1,1])$ has the product structure $(A\cap D_i)\times[-1,1]$ for each $i$.
We define the \textit{Whitehead graph $W(A,\DD)$ of $A$ with respect to $\DD$}
as the graph such that the vertex set is $\SS\cup\SS^{-1}$
and the edge set consists of arcs in $A\setminus (\DD\times(-1,1))$,
where  $A\cap (D_i\times\pm1)$ are identified with the vertices $a_i^{\pm1}$ in $W(A,\DD)$.
In other words, $W(A,\DD)$ is the graph obtained from
$(A\cup(\DD\times[-1,1]))\setminus\DD\times(-1,1)$
by collapsing each $D_i\times\{\pm1\}$ onto the vertices $a_i^{\pm1}$.
Suppose each loop $\gamma_j$ in $A=\gamma_1\cup\ldots\cup\gamma_r$
is equipped with an orientation.
Follow each loop $\gamma_j$, and whenever $\gamma_j$ intersects $D_i$, record 
$a_i$ or $a_i^{-1}$ according to whether the orientation of $\gamma_j$ coincides with the transverse
orientation of $D_i$ or not. Let $w_j$ be the 
word thus obtained. 
In this case, 
we say that $U=\{w_1,\ldots,w_r\}$ is \textit{realized by $A$} with respect to $\DD$.
As we are assuming that $A$ intersects $\DD$ minimally,
$U$ consists of cyclically reduced words.
Depending on the choice of the basepoint of $\gamma_j$, 
$w_j$ is determined up to cyclic conjugation.
Observe that $W(A,\DD)=W(U)$, 
and 
the notion of the connecting map is also defined for $W(A,\DD)$ by that of $W(U)$.
Moreover, for any other choice of disk structure $\DD'$ on $H$ which $A$ intersects transversely and minimally,
$W(A,\DD')=W(\phi(U))$ for some $\phi\in\aut(F)$.

For a simple loop $\gamma$ embedded in $H$ and $c,j>0$, we let $c\gamma^j$ denote the $1$--submanifold consisting of $c$ components, each of which is freely homotopic to $\gamma^j$. Consider a non-zero integral polynomial $f(x_1,\ldots,x_r)=\sum_{i,j\ge1} c_{ij} x_i^j$ where $c_{ij}\ge0$. Let $A$ be a $1$--submanifold of $H$ written as the union of disjoint loops $A=\gamma_1\cup\cdots\cup\gamma_r$. Then $f(A)=f(\gamma_1,\ldots,\gamma_r)$ denotes a $1$--submanifold in $H$ freely homotopic to $\coprod_{i,j\ge1} c_{ij}\gamma_i^j$. $f(A)$ is again an abusive notation in that $A$ is considered as an ordered tuple of loops, rather than just a set of loops. If a $1$--submanifold $A\subseteq H$ realizes $U\subseteq F$ with respect to a disk structure $\DD$, then $W(f(U)) = W(f(A),\DD)$.

\begin{rem}\label{rem:minimal}
When considering the Whitehead graph of $U\subseteq F$ or a $1$--submanifold $A\subseteq H$, we always require that $U$ consists of cyclically reduced words, and $A$ intersects $\DD$ transversely and minimally. If necessary, we achieve this by replacing words in $U$ by some cyclic conjugations or by freely homotoping loops in $A$. Note that certain additional conditions on $A$, such as $A\subseteq\partial H$, can possibly be lost by freely homotoping $A$.
\end{rem}

A properly embedded disk $D$ in $H$ is \textit{essential} if either $H\setminus D$ is connected or neither of the components of $H\setminus D$ is a $3$--cell. Let $A$ be a $1$--submanifold of $H$.  Suppose that for any $1$--submanifold $A'$ freely homotopic to $A$, $A'$ intersects any essential disk $D\subseteq H$. Then
 $A$ is said to be \textit{diskbusting}. Note that a set $U$ of cyclically reduced words in $F$ is diskbusting if and only if a $1$--submanifold $A\subseteq H$ realizing $U$ is diskbusting.

A vertex $v$ in a graph $\Gamma$ is a \textit{cut vertex} if $\Gamma\setminus\{v\}$ is not connected.
We say that a Whitehead graph $W(A,\DD)$ is \textit{minimal}
\footnote{Minimality of $W(A,\DD)$ should not be confused with the hypothesis that $A$ intersects $\DD$ transversely and minimally.}
if 
$|E(W(A,\DD))|\le |E(W(A',\DD'))|$
for any $1$--submanifold $A'$ freely homotopic to $A$ and 
any disk structure $\DD'$ which $A'$ intersects transversely and minimally.
$U\subseteq F$ is \textit{minimal} if $U$ is not equivalent to any $U'\subseteq F$ where the sum of the lengths of the words in $U'$ is smaller than that of $U$.
Using a well-known finite-time algorithm to determine whether $U\subseteq F$ is diskbusting or
not, one gets the following:

\begin{theorem}[\cite{Whitehead:1936p4475,Stong:1997p4326,Stallings:1999p3173}]
	\begin{enumerate}
	\item
If $U\subseteq F$ is minimal, independent and diskbusting,
then $W(U)$ is connected and does not have a cut vertex.
	\item
	Suppose $A\subseteq H$ is a diskbusting $1$--submanifold
	which intersects a disk structure $\DD$ on $H$ transversely and minimally,
	such that $W(A,\DD)$ is minimal.
	Then $W(A,\DD)$ is connected and does not have a cut vertex.\qed
	\end{enumerate}
	\label{thm:stallings}
\end{theorem}
    
\section{Simple Surgery}\label{sec:ss}
Zieschang proved a key fact on a minimal Whitehead graph of a geometric $1$--submanifold in $H$.
    
\begin{theorem}[\cite{Zieschang:1965p4314}; see also~\cite{Berge:2009p3422}]\label{thm:geometric}
	Suppose $A$ is a geometric $1$--submanifold of $H$ which intersects a disk structure $\DD$ transversely and minimally. If $W(A,\DD)$ is minimal, then $A$ is freely homotopic to some $1$--submanifold $A'\subseteq\partial H$ such that $A'$ intersects $\DD$ transversely and minimally. In particular, $W(A',\DD)=W(A,\DD)$.
\qed
\end{theorem}

A \textit{pairing} $\sim$ on a finite set $X$ will mean an equivalence relation on $X$ such that each equivalence class consists of precisely two elements. From now on, we will use the notation $I_1=(0,1]$ and $I_{-1}=[-1,0)$.

\begin{setting}[for Definition~\ref{defn:ss}]\label{setting:ss}
Let $A\subseteq H$ be a $1$--submanifold and $\DD=\{D_1,\ldots,D_n\}$ be a disk structure on $H$ such that $A$ intersects $\DD$ transversely and minimally.
Recall our convention that  $\DD\times [-1,1]=\cup_i (D_i\times[-1,1])$  denotes a closed regular neighborhood of $\DD$ equipped with a product structure so that $D_i$ is identified with $D_i\times 0$. In particular, if $v\in D_i$, then $v\times 1\in D_i\times1\subseteq H$ and  $v\times -1\in D_i\times-1\subseteq H$. We will also assume that $A$ is compatible with the product structure in the sense that $A\cap (D_i\times[-1,1]) = (A\cap D_i)\times[-1,1]$ for each $i$. 
For each $i$, suppose $|A\cap D_i|$ is even and $\sim_i$ is a pairing on $A\cap D_i$.
For each pair $v\sim_i v'$, let $\alpha_{vv'}^{+}\subseteq D_i\times I_1$ be an arc joining $v\times1$ and $v'\times 1$, and $\alpha_{vv'}^-\subseteq D_i\times I_{-1}$ be an arc joining $v\times-1$ and $v'\times-1$;
moreover, we assume that all the arcs in $\{\alpha_{vv'}^\pm\cobar\; v\sim_i v'\mbox{ for some }i\}$ are disjoint (see Fig.~\ref{fig:sseg1} and~\ref{fig:sseg2}).
Let $A'\subseteq H\setminus\DD$ be the $1$--submanifold obtained from $A$ if one replaces $A\cap (D_i\times[-1,1])$ by the arcs $\cup_{v\sim_i v'}\alpha_{vv'}^\pm$ for all $i$. 
\end{setting}

\begin{definition}\label{defn:ss}
In Setting~\ref{setting:ss}, assume further the following:
\enumir
\begin{enumerate}
\item
for each $1\le i\le n$ and $\epsilon=\pm1$, the intersection between each component of $A'$ and $D_i\times I_\epsilon$ is either connected or empty;
\item
for some component $\gamma$ of $A'$, $\gamma\cap(\DD\times[-1,1])$ has more than two components.
\end{enumerate}
Then we say that $(\DD,\{\sim_i\})$ \textit{determines a simple surgery on $A$.}
When $\sim_i$ needs not be explicitly notated, we simply say \textit{$A$ admits a simple surgery (with respect to $\DD$)}.
Also, $A'$ is said to be \textit{obtained by the simple surgery $(\DD,\{\sim_i\})$ on $A$}.
\enumia
\end{definition}

\begin{exmp}\label{exmp:ss1}
Let $A\subseteq H$ be the $1$--submanifold denoted as the bold curve in Fig.~\ref{fig:sseg1} (a). Write $D_1\cap A = \{v,v'\}$ and $D_2\cap A=\{w,w'\}$. Define $v\sim_1 v'$ and $w\sim_2 w'$, and consider $\alpha_{vv'}^\pm$ and $\alpha_{ww'}^\pm$ as in Setting~\ref{setting:ss}.
Put $\DD=\{D_1,D_2\}$.
Figure~\ref{fig:sseg1} (b) shows $A'\subseteq H$ which is obtained by substituting $\alpha_{vv'}^\pm$ and $\alpha_{ww'}^\pm$ for $A\cap (\DD\times[-1,1])$.
Note that  $A'\cap (\DD\times[-1,1])$ has four components $\{\alpha^+_{vv'},\alpha^-_{vv'},\alpha^+_{ww'},\alpha^-_{ww'}\}$.
It follows that $A$ admits the simple surgery $(\DD,\{\sim_1,\sim_2\})$. 
\begin{figure}[htb!]
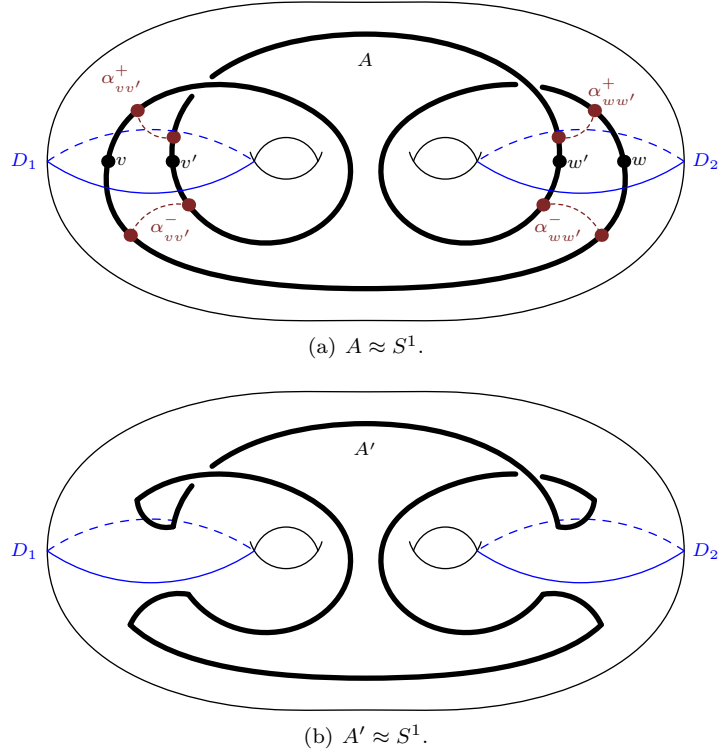

	\subfigure[$A\approx S^1$.]{
	\includegraphics[]{fig/figsseg1.1}
	}
	\subfigure[$A'\approx S^1$.]{
	\includegraphics[]{fig/figsseg1.2}
	}
	\caption{$A'$ is obtained by a simple surgery on $A$.
	\label{fig:sseg1}}
\end{figure}
\end{exmp}

\begin{exmp}\label{exmp:ss2}
The bold curves in Fig.~\ref{fig:sseg2} (a) denotes $A\subseteq H$.
Put $D_1\cap A = \{u,u',v,v'\}$ and $D_2\cap A=\{w,w'\}$ as shown in the figure. 
Set $\DD=\{D_1,D_2\}$.

Define $u\sim_1 u',v\sim_1 v'$ and $w\sim_2 w'$. 
Replace suitable segments on $A$ by $\alpha_{uu'}^\pm,\alpha_{vv'}^\pm$ and $\alpha_{ww'}^\pm$ as in Setting~\ref{setting:ss}, to obtain $A'=\gamma_1\coprod\gamma_2\coprod \gamma_3$; see Fig.~\ref{fig:sseg2} (b).
Note that condition (ii) of Definition~\ref{defn:ss} is violated, since $\gamma_i\cap(\DD\times[-1,1])$ has exactly two components for each $i=1,2,3$.

Consider a different pairing $u\sim_1 v',u'\sim_1 v$ and $w\sim_2 w'$.
Following Setting~\ref{setting:ss}, replace segments in $A\cap (\DD\times[-1,1])$ accordingly to obtain $A''\approx S^1$ described in Fig.~\ref{fig:sseg2} (c).
Condition (i)  of Definition~\ref{defn:ss} then fails, since $A''\cap (D_1\times I_1)$
and $A''\cap (D_1\times I_{-1})$ are both disconnected. In this way, one can see that $A$ does not admit any simple surgery with respect to $\DD$.

\begin{figure}[htb!]
	\subfigure[$A\approx S^1\coprod S^1\coprod S^1\coprod S^1$.]{
	\includegraphics[]{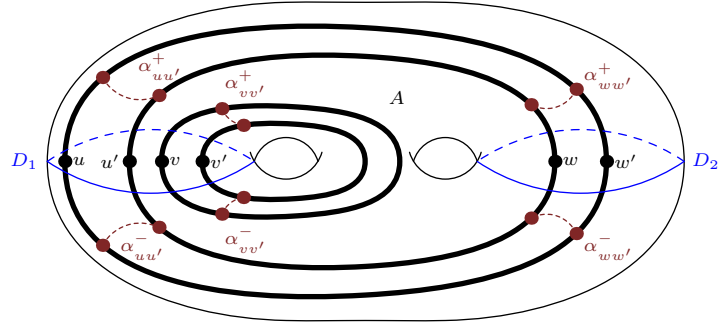}
	}
	\subfigure[$A'\approx S^1\coprod S^1\coprod S^1$.]{
	\includegraphics[]{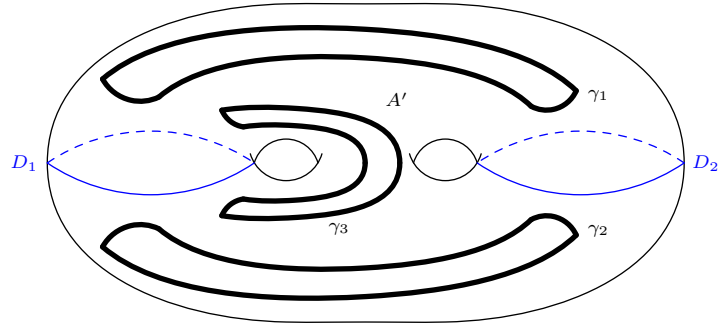}
	}
	\subfigure[$A''\approx S^1$.]{
	\includegraphics[]{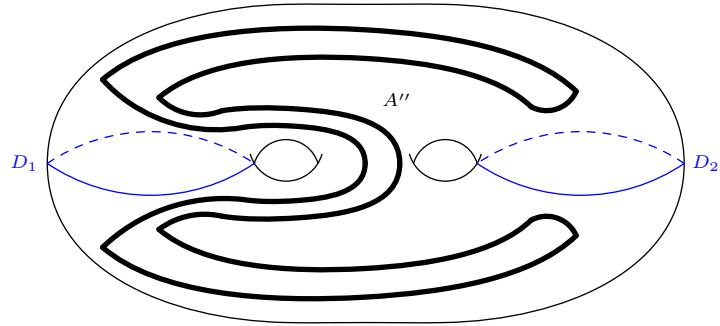}
	}
\caption{$A$ does not admit a simple surgery with respect to $\DD=\{D_1,D_2\}$.
	\label{fig:sseg2}}
\end{figure}
\end{exmp}

Whether $A\subseteq H$ admits a simple surgery or not can be detected by looking at a corresponding Whitehead graph.

\begin{prop}\label{prop:ss}
Let $A$ be a $1$--submanifold of $H$ 
which intersects a disk structure $\DD$ transversely and minimally.
Let $\sigma$ denote the connecting map associated with $\Gamma=W(A,\DD)$.
Then  $A$ admits a simple surgery with respect to $\DD$ if and only if $\Gamma=\cup_h C_h$ for some $C_h$'s satisfying the following:
\enumir
\begin{enumerate}
\item
each $C_h$ is a simple cycle,
\item
$C_h$ and $C_h'$ do not have a common edge whenever $h\ne h'$,
\item
if  $e$ and $e'$ are edges of some $C_h$ intersecting at a vertex $v$,
then $\sigma(e,v)$ and $\sigma(e',v)$ are edges of some $C_{h'}$,
\item
at least one $C_h$ is not a bigon.
\end{enumerate}
\enumia
\end{prop}

In condition (iii), $h$ and $h'$ can possibly be the same.

\begin{proof}
Write $\DD=\{D_1,\ldots,D_n\}$.
We let $\Phi\co (A\cup (\DD\times[-1,1]))\setminus(\DD\times(-1,1))\rightarrow\Gamma$
denote the map collapsing $D_i\times\pm1$ onto the vertices $a_i^{\pm1}$.

($\Rightarrow$) 
Let $A'$ be obtained from $A$ by the simple surgery $(\DD,\{\sim_i\})$.
$\Phi$ maps each loop in $A'$ onto a simple cycle in $\Gamma$,
by the definition of a simple surgery. Hence, we have the condition (i).  
Let $\{C_1,C_2,\ldots, C_t\}$ denote the set of simple cycles in $\Gamma$ thus obtained.
(ii) follows from the fact that $A'$ consists
of disjoint loops. 
For (iii), let $e,e'\in E(C_h)$ intersect with, say, $a_i\in V(\Gamma)$.
Then there exist two arcs $\alpha,\alpha'$ in $A\setminus (\DD\times(-1,1))$
and $v,v'\in A\cap D_i$
such that
$v\sim_i v'$,
$v\times1\in \alpha,v'\times 1\in \alpha'$
and,
$e$ and $e'$ are the images of $\alpha$ and $\alpha'$
by $\Phi$.
Let $f$ and $f'$ be the edges in $\Gamma$
which are the images of the arcs in $A\setminus (\DD\times(-1,1))$
intersecting the disk $D_i\times-1$ at $v\times-1$ and at $v'\times-1$, respectively.
By the definition of $\sigma$, we have $\sigma(e,v)=f$ and $\sigma(e',v)=f'$.
Since $v\sim_i v'$, $f$ and $f'$ belong to some $C_{h'}$, implying (iii).
For (iv), one has only to consider a loop $\gamma$ in $A'$ that intersects more than two of $D_i\times I_{\pm1}$'s.

($\Leftarrow$) 
Choose any $D_i$ and $v,v'\in A\cap D_i$ such that $v\ne v'$.
Let $\gamma$ and $\gamma'$ be (possibly the same) loops in $A$ 
intersecting $D_i$ at $v$ and $v'$, respectively.
Let $e_+,e_-,e'_+$ and $e'_-$ denote the edges of $\Gamma$ corresponding to the components 
of $A\setminus(\DD\times(-1,1))$
containing $v\times1$, $v\times-1$, $v'\times 1$ and $v'\times-1$, respectively.
Define $v\sim_i v'$ if $e_+$ and $e'_+$ happen to be consecutive edges of some $C_h$;
by (iii), this occurs if and only if
$e_-=\sigma(e_+,a_i)$ and $e_-'=\sigma(e'_+,a_i)$ are consecutive in some $C_{h'}$.
By (ii), $\sim_i$ is a pairing on $A\cap D_i$.
Let $A'$ be obtained from $A$ and from the pair $(\DD,\{\sim_i\})$, by the process described in Setting~\ref{setting:ss}.
Whenever $v\sim_i v'$ for some $i$, there is an arc in $A'$ joining $v\times1$ and $v'\times1$, and another arc in $A'$ joining $v\times-1$ and $v'\times-1$.
$\Phi$ sends each loop in $A'$ to some $C_h$ by the definition of $\{\sim_i\}$.
Each loop in $A'$ does not intersect $D_i\times I_\epsilon$
more than once for any $i$ and $\epsilon=\pm1$, by the condition (i).
Condition (iv) implies that at least one loop in $A'$ intersects more than two of $D_i\times I_{\pm1}$'s.
Hence, $(\DD,\{\sim_i\})$ determines a simple surgery on $A$.
\end{proof}

For a planar graph $\Gamma\subseteq S^2$,
a \textit{bigon neighborhood} of $\Gamma$ is $\cup_{e\in E(\Gamma)} N_e\subseteq S^2$,
where 
\enumir
\begin{enumerate}
\item
each $N_e$ is a $2$--cell such that $e$ is properly embedded in $N_e$, and
\item
$N_e\cap N_{e'}=e\cap e'$ for $e\ne e'$.
\end{enumerate}
\enumia
For example, see Fig.~\ref{fig:figbigon}.

\begin{figure}[htb!]
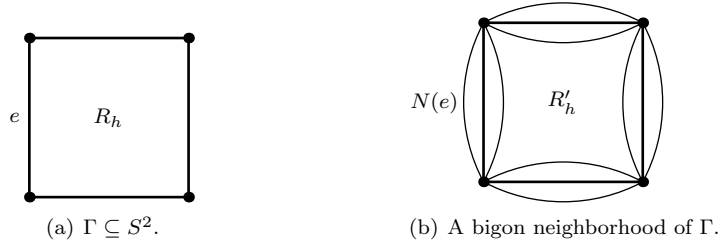

	\subfigure[$\Gamma\subseteq S^2$.]{
	\includegraphics[]{fig/figbigon.0}
	}
	\hspace{.9in}
	\subfigure[A bigon neighborhood of $\Gamma$.]{
	\includegraphics[]{fig/figbigon.1}
	}
	\caption{An example of a bigon neighborhood. Note that $R_h'\subseteq R_h$.
	\label{fig:figbigon}}
\end{figure}

\begin{definition}\label{defn:wss}
	Consider a $1$--submanifold $A=\gamma_1\cup\cdots\cup\gamma_r\subseteq H$ where each $\gamma_i$ is a loop.
If there exists a non-zero integral polynomial of the form $f(x_1,\ldots,x_r) = \sum_{i,j\ge1} c_{ij} x_i^j$ where $c_{ij}\ge0$ such that $f(A)=\coprod_{i,j\ge1}c_{ij}\gamma_i^j$ admits a simple surgery with respect to a disk structure $\DD$ on $H$, then we say that \textit{$A$ weakly admits a simple surgery (with respect to $\DD$)}.
\end{definition}

Proposition~\ref{prop:g implies ss} and~\ref{prop:p and ss} are the key steps for the proof of Theorem~\ref{thm:g implies p} and~\ref{thm:vg implies vp}.

\begin{prop}\label{prop:g implies ss}
Suppose $A$ is a geometric and diskbusting $1$--submanifold of $H$ which intersects a disk structure $\DD$ transversely and minimally, such that $W(A,\DD)$ is minimal. Then $A$ weakly admits a simple surgery with respect to $\DD$.
\end{prop}

\begin{proof}
	By Theorem~\ref{thm:geometric}, we may assume that $A\subseteq \partial H$.
	Write $\DD=\{D_1,\ldots,D_n\}$ and $\Gamma=W(A,\DD)$.
	Let $N$ be a closed regular neighborhood of $A$ in $\partial H$.
	$B=\partial N$ denotes the boundary of $N$ as a $2$--submanifold of $\partial H$.
	Note that $B = A+A$; that is, $B$ consists of two $1$--submanifolds each of which is freely homotopic to $A$.
	In particular, $B$ intersects $\DD$ transversely and minimally.
Since $A$ is diskbusting, $\Gamma$ is connected (Theorem~\ref{thm:stallings}); in particular, $\Gamma$ does not have any isolated vertices.
Hence, one can write $S^2\setminus \Gamma = \coprod_h R_h$ where each $R_h$ is an open disk.
	Choose a bigon neighborhood $N_1\subseteq S^2$ of $\Gamma\subseteq S^2$, so that
	$S^2\setminus N_1 = \coprod_h R_h'$ for some $R_h'\subseteq R_h$ (Fig.~\ref{fig:figbigon}).
Write $\Gamma'=W(B,\DD)$. Then we have an edge decomposition
 $\Gamma'=\cup_h \partial\overline{R_h'}\subseteq S^2$.
	Define a pairing $\sim_i$ on $B\cap D_i = B\cap \partial D_i$
	by $v\sim_i v'$ if $v$ and $v'$ are the endpoints of some interval (denoted as $\alpha_{vv'}$)
	in $\partial D_i\setminus\mathrm{int}(N)$.
	
	We claim that $(\DD,\{\sim_i\})$ determines a simple surgery on $B$.
	As in the proof of Proposition~\ref{prop:ss},
	choose a product structure on $\DD\times[-1,1]$ 
	such that $B\cap (D_i\times[-1,1])$ is identified with $(B\cap D_i)\times[-1,1]$ for each $i$.
	In particular, there exist homeomorphisms $D_i\times0\rightarrow D_i\times \pm1$
	which are compatible with the product structure.
	Whenever $v\sim_i v'$, choose $\alpha_{vv'}^+\subseteq D_i\times1$
	and $\alpha_{vv'}^-\subseteq D_i\times-1$  
	to be the image of $\alpha_{vv'}\subseteq  D_i\times 0$.
	Let 
	$B' = (B\setminus (\DD\times(-1,1)))\cup(\cup_{v\sim_i v'}\alpha_{vv'}^\pm)$.
	Then
	$\Gamma'$ is obtained from $B'$ by collapsing $\alpha_{vv'}^\pm$ onto the vertices $a_i^{\pm1}$
	for each $v$ and $v'$ in $D_i$ such that $v\sim_i v'$. 
	By this process, each loop in $B'$ collapses onto 
 	the cycle $C_h = \partial\overline{R_h'}$ for some $h$.
	In the below, let us verify the conditions (i) through (iv) of Proposition~\ref{prop:ss} with regard to the decomposition $\Gamma'=\cup_h C_h$.
	
	The condition (ii) of Proposition~\ref{prop:ss} is obvious, since $\partial\overline{R_h'}$ and $\partial\overline{R_{h'}'}$ intersect only at vertices of $\Gamma'$. Condition (iii) easily follows from the proof of ($\Rightarrow$) in Proposition~\ref{prop:ss}. By Theorem~\ref{thm:stallings}, $\Gamma$ is a connected graph without a cut vertex. By Lemma~\ref{lem:graph}, each $\partial\overline{R_h}$ is a simple cycle. This implies $C_h=\partial\overline{R_h'}\approx\partial\overline{R_h} $ is a simple cycle, and hence, we have (i).

	The condition (iv) fails only when each $\partial\overline{R_h'}$, and hence each $\partial\overline{R_h}$ also,
	is a bigon.
	Since each edge is shared by two regions, the number of edges in $\Gamma$ must then be the same
	as the number of regions in $S^2\setminus\Gamma$.
	Considering $S^2$ as a CW-complex having the connected graph $\Gamma$ as its $1$--skeleton,
	we would have	$2=\chi(S^2)=|V(\Gamma)|$.  
	This contradicts to the fact that the genus of $H$ is larger than $1$.
\end{proof} 

\begin{lem}\label{lem:graph}
	Let $\Gamma\subseteq S^2$ be a connected graph without a cut vertex.
	If $R$ is a component of $S^2\setminus\Gamma$, then $\partial\overline{R}$ is a simple closed curve.
\end{lem}
\begin{proof}
	Let $N(\Gamma)$ denote a closed regular neighborhood of $\Gamma$.  
	Denote by $T$ the component of $S^2\setminus N(\Gamma)$ such that $T\subseteq R$.
	Since $\Gamma$ is connected, so is $N(\Gamma)$. Hence, $N(\Gamma)$ is a punctured sphere, and $T$ is an open disk.
	By the deformation retract $N(\Gamma)\rightarrow \Gamma$,
	$\partial\overline{T}$ maps to $\partial\overline{R}$.
	So, $\partial\overline{R}$ is a closed curve and $R$ is an open disk.
	There exists a polygonal disk $\overline{Q}$ and a quotient map $q\co \overline{Q}\rightarrow\overline{R}$
	such that $\mathrm{int}(q)\co Q\rightarrow R$ is a homeomorphism
	and $\partial q\co\partial\overline{Q}\rightarrow\partial\overline{R}$ is a graph map, where
	$\mathrm{int}(q)$ and $\partial q$ are restrictions of $q$.
	We have only to show that $\partial q$ is $1-1$.
	Suppose $q(x)=q(y)$ for some $x\ne y\in (\partial \overline{Q})^{(0)}$.
	Pick a properly embedded arc $\alpha\subseteq \overline{Q}$ joining $x$ and $y$,
	and
	write $Q\setminus\alpha = Q_1\cup Q_2$ (Fig.~\ref{fig:lem:graph} (a)).
	Since $q(\alpha)\approx S^1$, we can write $S^2\setminus q(\alpha)=A_1\cup A_2$
	such that $A_i$ is an open disk and $q(Q_i)= R\cap A_i$.
	So, $q(\alpha)$ separates $q(\partial\overline{Q_1}\setminus\alpha)\setminus q(x)$ from
	$q(\partial \overline{Q_2}\setminus\alpha)\setminus q(x)$ (Fig.~\ref{fig:lem:graph} (b)).
	Since $q(\alpha)\cap \Gamma = q(\alpha)\cap\partial\overline{R} = q(x) = q(y)$,
	$q(x)$ separates $\Gamma$. This is a contradiction.
\end{proof}

\begin{figure}[htb!]
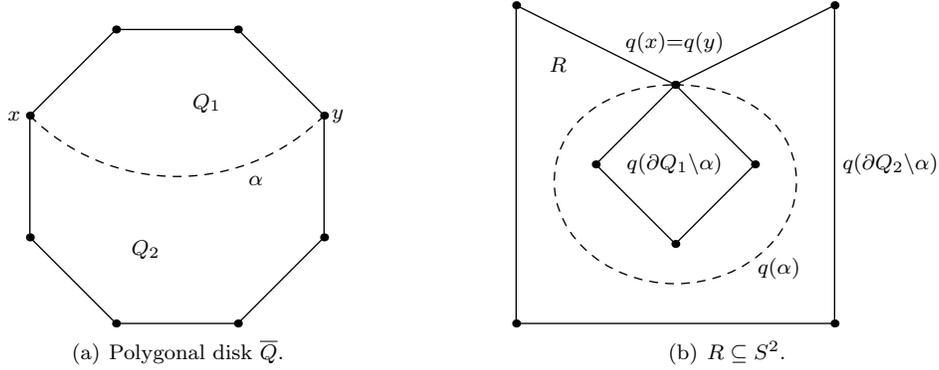

	\subfigure[Polygonal disk $\overline{Q}$.]{
	\includegraphics[]{fig/figcutvertex.1}
	}
\hfill
	\subfigure[$R\subseteq S^2$.]{
	\includegraphics[]{fig/figcutvertex.2}
	}
	\caption{Proof of Lemma~\ref{lem:graph}.
	\label{fig:lem:graph}}
\end{figure}

\begin{prop}\label{prop:p and ss}
Let $U\subseteq F$ be an  independent, diskbusting set of root-free, cyclically reduced words.
If a $1$--submanifold $A\subseteq H$ realizes $U$ with respect to a disk structure $\DD$, then the following are equivalent.
\begin{enumerate}
\item
$U$ is polygonal.
\item
$A$ weakly admits a simple surgery with respect to $\DD$.
\end{enumerate}
\end{prop}

\begin{proof}
Write $U = \{w_1,\ldots,w_r\}$ and  $A=\gamma_1\cup\ldots\cup\gamma_r$ so that
$\gamma_i$ realizes $w_i$ with respect to $\DD$.
Recall our convention that $A$ intersects $\DD$ transversely and minimally (Remark~\ref{rem:minimal}).

\textbf{(1)$\Rightarrow$(2):}
Suppose there exists a $U$--polygonal surface $S$ and an associated immersion $\phi\co S^{(1)}\rightarrow\cay(F)/F$ satisfying $\chi(S)<m(S)$.
For $i,j\ge1$, let $c_{ij}$ denote the number of polygonal disks $P$ on $S$ such that the composition of immersions $\partial P\to S^{(1)}\stackrel{\phi}\to\cay(F)/F$ reads $w_i^j$.
In particular, $m(S) = \sum_{i,j} c_{ij}$.
One can write $S = \coprod_i \coprod_j \coprod_{k=1}^{c_{ij}}P_{ijk}/\!\!\sim$
for some side-pairing  $\sim$ such that 
each $P_{ijk}$ is a polygonal disk whose boundary reads $w_i^j$.
Put $f(x_1,\ldots,x_r) = \sum_{i,j} c_{ij}x_i^j$
and $\Gamma = W(f(U))$.
Let $\sigma$ denote the connecting map associated with $\Gamma$.
By Remark~\ref{rem:connecting map} (2),
there is a natural $1-1$ correspondence $\rho$ between
the edges of $\Gamma$
and the corners (\textit{i.e.}, two adjacent edges) of $\coprod_{i,j,k}P_{ijk}$.
For instance, 
a vertex $v\in\partial P_{ijk}$ at which 
an $a_1$-edge is incoming and
an $a_2$-edge is outgoing
corresponds to an edge of $\Gamma$
joining $a_1$ and $a_2^{-1}$ (Fig.~\ref{fig:prop:p and ss 1}).

\begin{figure}[htb!]
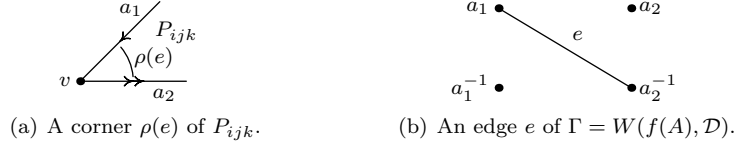

	\subfigure[A corner $\rho(e)$ of $ P_{ijk}$.]{
	\includegraphics[]{fig/figpolygonalandss.0}
	}
	\hspace{.45in}
	\subfigure[An edge $e$ of $\Gamma=W(f(A),\DD)$.]{
	\includegraphics[]{fig/figpolygonalandss.1}
	}
	\caption{The $1-1$ correspondence $\rho$.
	\label{fig:prop:p and ss 1}}
\end{figure}

Write $S^{(0)}= \{v_1,\ldots,v_t\}$.
$\link(v_h)\approx S^1$ denotes the link of $v_h$ in $S$.
$\rho^{-1}$ maps each $\link(v_h)$ to some cycle $C_h\subseteq\Gamma$. 
Since $\phi\co S^{(1)}\rightarrow \cay(F)/F$ is locally injective, each $C_h$ is a simple cycle.
This proves the condition (i) of Proposition~\ref{prop:ss}.
The condition (ii) follows from the fact
that $\rho$ is a $1-1$ correspondence.

Suppose $e$ and $e'$ are consecutive edges  in some $C_h$.
$e$ and $e'$ correspond to an adjacent pair of edges (corners) in $\link(v_h)$.
Without loss of generality, 
assume $a_q\in \partial e\cap\partial e'\subseteq V(\Gamma)$,
and let $v_{h'}\in S^{(0)}$ be the other endpoint of
the $a_q$-edge incoming at $v_h$  (Fig.~\ref{fig:prop:p and ss 2} (a)).
$\rho(e)$ is the corner of some $P_{ijk}$ at $v_h$,
and similarly, $\rho(e')$ is the corner of some $P_{i'j'k'}$ at $v_h$.
From the definition of $\sigma$, one can see that the corners of $P_{ijk}$ and $P_{i'j'k'}$ at $v_{h'}$
correspond to $f=\sigma(e,a_q)$ and $f'=\sigma(e',a_q)$, respectively.
$f$ and  $f'$ belong to some
$C_{h'}$ as consecutive edges (Fig.~\ref{fig:prop:p and ss 2} (b)). 
This proves the condition (iii) of Proposition~\ref{prop:ss}.

\begin{figure}[htb!]
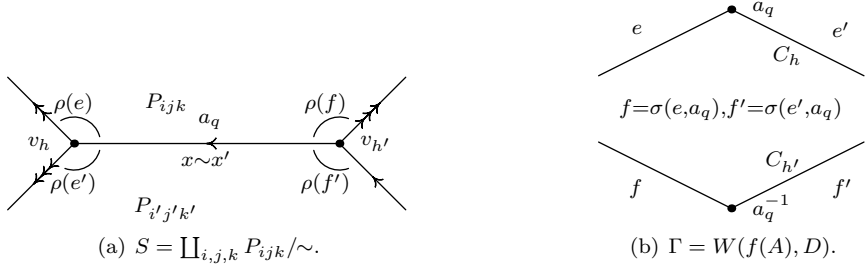

	\subfigure[$S=\coprod_{i,j,k} P_{ijk}/\!\!\sim$.]{
	\includegraphics[]{fig/figpolygonalandsss.1}
	}
	\hspace{.45in}
	\subfigure[$\Gamma=W(f(A),D)$.]{
	\includegraphics[]{fig/figpolygonalandsss.2}
	}
	\caption{Visualizing the connecting map $\sigma$.
	\label{fig:prop:p and ss 2}} 
\end{figure}

Now we will use the assumption $\chi(S)<m(S)$ as follows.
The number of vertices, edges and faces in $S=\coprod_{i,j,k} P_{ijk}/\!\!\sim$ are
$t$, $\sum_{i,j} c_{ij}|w_i^j|/2$ and $\sum_{i,j}c_{ij}$, respectively. Hence,
\begin{eqnarray*}
 \sum_{i,j}c_{ij} = m(S) &>& \chi(S) = t - \sum_{i,j} c_{ij} |w_i^j|/2 + \sum_{i,j}c_{ij},\\
2t &<& \sum_{i,j} c_{ij} |w_i^j| = |E(\Gamma)|=\sum_{h=1}^t |E(C_h)|.
\end{eqnarray*}

It follows that $|E(C_h)|>2$ for some $h$; that is, the condition (iv) of Proposition~\ref{prop:ss} is satisfied.

\textbf{(2)$\Rightarrow$(1):}
Suppose $f(x_1,\ldots,x_r) = \sum_{i,j\ge1}c_{ij} x_i^j$
is a non-zero integral polynomial with $c_{ij}\ge0$ such that there is a simple surgery on
$f(A)=f(\gamma_1,\ldots,\gamma_r)$.
Write $\Gamma=W(f(A),\DD)=\cup_{h=1}^t C_h$ so that the conditions in Proposition~\ref{prop:ss}
are satisfied.
For each $i,j$ and $1\le k\le c_{ij}$, take a polygonal disk $P_{ijk}$ equipped with an immersion
$\partial P_{ijk}\rightarrow\cay(F)/F$ reading $w_i^j$; so, each edge in $\coprod_{i,j,k} \partial P_{ijk}$
carries a label by $\SS=\{a_1,\ldots,a_n\}$ as well as an orientation
induced from the orientations of the edges in $\cay(F)/F$.
There is a $1-1$ correspondence $\rho$ between
the edges of $\Gamma$ and 
the corners of $\coprod_{i,j,k} P_{ijk}$ as in the proof of
(1)$\Rightarrow$(2).
Let $1\le q\le n$, and consider two $a_q$-edges $x\in E(\partial P_{ijk})$ and $x'\in \partial E(P_{i'j'k'})$.
We declare that $x\sim x'$ if 
$\rho^{-1}$ sends the corners at the terminal vertices of $x$ and $x'$
to consecutive edges $e$ and $e'$ in some $C_h$ 
(Fig.~\ref{fig:prop:p and ss 2} (a)).
By the condition (ii) of Proposition~\ref{prop:ss},
$e\in E(\Gamma)$ and $a_q\in \partial e$ uniquely determines $e'\in E(\Gamma)$ such that $e$ and $e'$ are consecutive at $a_q$ in some $C_h$.
By the condition (iii), 
the other corners of $x$ and $x'$ correspond to
consecutive edges
$f=\sigma(e,a_q)$
and $f'=\sigma(e',a_q)$ in some $C_{h'}$.
Hence, $\sim$ defines a side-pairing on $\coprod_{i,j,k} P_{ijk}$
that matches the labels and the orientations of the edges; moreover,
$\rho$ determines a $1-1$ correspondence between the links of the vertices in the closed surface $S=\coprod_{i,j,k}P_{ijk}/\!\!\sim$ and the cycles $C_1,C_2,\ldots,C_t$.
Since each $C_h$ is simple, 
there is an immersion $S^{(1)}\rightarrow \cay(F)/F$ induced by
the given immersion  $\coprod_{i,j,k}\partial P_{ijk}\rightarrow\cay(F)/F$.
Note that $m(S)=\sum_{i,j}c_{ij}$.
As in the proof of (1)$\Rightarrow$(2),
\[\chi(S)= t - |E(\Gamma)|/2 + m(S)
		= t - \sum_{h=1}^t |E(C_h)|/2 + m(S).\]
By the condition (iv) of Proposition~\ref{prop:ss},
$|E(C_h)|>2$ for some $h$. Therefore, $\chi(S)<m(S)$.
\end{proof}

\begin{theorem}\label{thm:g implies p}
Suppose $U\subseteq F$ is an  independent, diskbusting set of root-free, cyclically reduced words.
If $U$ is geometric, then $U$ is equivalent to a polygonal set of words in $F$.
\end{theorem}

\begin{proof}
We may assume that $U$ is minimal, by applying an automorphism of $F$ if necessary. Let $U=\{w_1,\ldots,w_r\}$ be realized by $A\subseteq H$ with respect to a disk structure $\DD$. By Theorem~\ref{thm:geometric}, we can choose $A$ so that $A\subseteq \partial H$ and $W(A,\DD)$ is minimal. Proposition~\ref{prop:g implies ss} implies that $A$ weakly admits a simple surgery with respect to $\DD$. Hence, $U$ is polygonal by Proposition~\ref{prop:p and ss}.
\end{proof}

Now we consider virtually polygonal words (Definition~\ref{defn:vp}). Note that polygonal words are virtually polygonal. The converse is not true. Actually, there exists a non-polygonal word which becomes polygonal only after an automorphism of $F$ is applied. An example given in~\cite{Kim:2009p3867} is $w=abab^2ab^3\in F_2=\langle a,b\rangle$. While an elementary argument shows that $w$ is not polygonal, the automorphism $\phi$ defined by $\phi(a)=ab^{-2}$ and $\phi(b)=b$ maps $w$ to a Baumslag--Solitar relator $w'= a(a^2)^b$. Any Baumslag--Solitar relator $a^p(a^q)^b$ is polygonal for $pq\ne0$~\cite{Kim:2009p3867}.

\begin{definition}\label{defn:virtual ss}
A 1-submanifold $A\subseteq H$ \textit{virtually admits a simple surgery} if there exists a finite cover $p:H'\rightarrow H$ such that $p^{-1}(A)\subseteq H'$ weakly admits a simple surgery.
\end{definition}

\begin{prop}\label{prop:virtual p and ss}
Let $w\in F$ be a root-free and diskbusting word realized by a loop $\gamma\subseteq H$. Then $w$ is virtually polygonal if and only if $\gamma$ virtually admits a simple surgery.
\end{prop}

\begin{proof}
Let $F'$ be a finite-index subgroup of $F$ and $\SS'$ be a free basis for $F'$. Denote by $\hatw_{F'}$ a transversal for $\underw_{F'}$, so that $\hatw_{F'}$ is an independent set of cyclically reduced words written in $\SS'$. $\hatw_{F'}$ is realized by $p^{-1}(\gamma)$ with respect to some disk structure $\DD'$ on $H'$. By Proposition~\ref{prop:p and ss}, $\hatw_{F'}$ is polygonal with respect to $\SS'$ if and only if $p^{-1}(\gamma)$ weakly admits a simple surgery with respect to $\DD'$.
\end{proof}

\begin{theorem}\label{thm:vg implies vp}
A diskbusting and virtually geometric word in $F$ is virtually polygonal.
\end{theorem}

\begin{proof}
	Let $w$ be diskbusting and virtually geometric. We may assume $w$ is root-free and cyclically reduced. Let $\gamma\subseteq H$ realize $w$ with respect to a given disk structure $\DD$. Suppose $p:H'\rightarrow H$ is a finite cover such that $p^{-1}(\gamma)$ is freely homotopic to $A\subseteq\partial H'$. It is elementary to see that $A$ is also diskbusting in $H'$. By Proposition~\ref{prop:g implies ss}, $A$ weakly admits a simple surgery with respect to some disk structure $\DD'$ on $H'$. By Proposition~\ref{prop:virtual p and ss}, $w$ is virtually polygonal.
\end{proof}

Theorem~\ref{thm:vp} underlines the importance of virtual polygonality.

\begin{theorem}\label{thm:vp}
If a root-free word $w\in F$ is virtually polygonal, then $D(w)$ contains a hyperbolic surface group.
\end{theorem}

\begin{proof}
	We may assume $w$ is cyclically reduced by applying an automorphism of $F=\langle \SS\rangle$ if necessary.
    Let $\gamma_w\subseteq\cay(F)/F$ read $w$.
	Let $F'$ be a finite-index subgroup of $F$ such that
	a transversal $\hatw_{F'}$ for $\underw_{F'}$ is polygonal
	with respect to some free basis $\SS'$ for $F'$.
	Recall that $X_{\SS'}(\hatw_{F'})$ denote the 2--dimensional CW-complex
	obtained by taking two copies of $\cay_{\SS'}(F')/F'$ and gluing cylinders along
the loops reading $\hatw_{F'}$ considered as a set of words written in $\SS'$.
	By Theorem~\ref{thm:polygonal},
	$D(\hatw_{F'})=\pi_1(X_{\SS'}(\hatw_{F'}))$ contains a hyperbolic surface group.
	As was introduced in Section~\ref{sec:preliminary}, 
	$Y(w,F')$ denotes the finite cover of $X_\SS(w)$
	obtained by taking two copies of $\cay(F)/F'$ and gluing cylinders along the copies of the elevations of $\gamma_w$.
	Since the homotopy equivalence $\cay(F)/F'\rightarrow \cay_{\SS'}(F')/F'$
	maps each elevation of $\gamma_w$
	to a loop realizing an element in $\hatw_{F'}$,
	we have a homotopy equivalence	
	 $Y(w,F')\simeq X_{\SS'}(\hatw_{F'})$.
	Hence $D(w)\ge \pi_1(Y(w,F')) = D(\hatw_{F'})$ and so, $D(w)$ contains a hyperbolic surface group.
\end{proof}

\begin{cor}[Gordon--Wilton~\cite{Gordon:2009p360}]\label{cor:vp}
If $w\in F$ is root-free, virtually geometric and diskbusting, then $D(w)$ contains a hyperbolic surface group.\qed
\end{cor}

Converses of Theorem~\ref{thm:g implies p} and~\ref{thm:vg implies vp}
do not hold. Actually, we will prove the following theorem in Section~\ref{sec:nvg}.

\begin{theorem}\label{thm:p but not vg}
	There exist polygonal words which are not virtually geometric.
\end{theorem}

Virtual geometricity does not imply geometricity, in general. 
For example, $w=a^2b^{-1}ab\subseteq F_2=\langle a,b\rangle$ 
is virtually geometric, but not geometric~\cite{Gordon:2009p360}.
It is not known whether virtual polygonality is strictly weaker than
polygonality up to $\aut(F)$.

\begin{que}\label{que:p and vp}
Is a virtually polygonal word equivalent to a polygonal word?
\end{que}

\begin{tilingconjecture}[\cite{Kim:2009p3867}]\label{conj:tiling}
A minimal diskbusting word in $F$ is polygonal.
\end{tilingconjecture}

While the Tiling Conjecture is not resolved yet, we propose a weaker conjecture. 

\begin{vtilingconjecture}\label{conj:vtiling}
	A diskbusting word in $F$  is virtually polygonal.
\end{vtilingconjecture}

By Theorem~\ref{thm:vp}, the Virtual Tiling Conjecture would suffice to settle the Gromov's conjecture for $D(w)$.

\begin{rem}\label{rem:summary}
Let $w\in F$ be
a diskbusting and root-free word 
realized by $\gamma\subseteq H$.
Consider the following hypotheses on $w$ and $\gamma$:
\begin{itemize}
	\item
	* : No further hypothesis on $w$.
	\item
	(V)Geom : $w$ is (virtually) geometric.
	\item
	WSS : $\gamma$ weakly admits a simple surgery.
	\item
	VSS : $\gamma$ virtually admits a simple surgery.
	\item
	EPoly: $w$ is equivalent to a polygonal word.
	\item
	VPoly : $w$ is virtually polygonal.
	\item
	DSurf : $D(w)$ contains a hyperbolic surface group.
\end{itemize}
Note that none of the above hypotheses are dependent on the choice of a free basis for $F$, or equivalently, on the choice of a disk structure on $H$. We can summarize the content of this section as the following diagram:
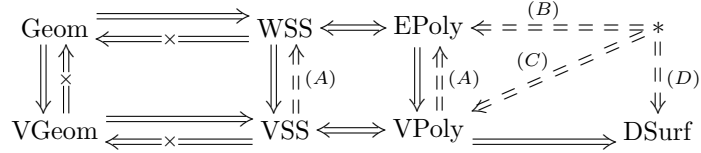
\[
\xymatrix{
\mathrm{Geom}\ar@{=>}@<1ex>[rr]\ar@{=>}@<-1ex>[d]
 &&\mathrm{WSS}\ar@{<=>}[r]\ar@{=>}@<-1ex>[d]\ar@{=>}@<1ex>[ll]|\times
&\mathrm{EPoly}\ar@{=>}@<-1ex>[d]
 &&\mathrm{\ast}\ar@{==>}[ll]_{(B)}\ar@{==>}[d]^{(D)}\ar@{==>}[lld]_{(C)}\\
\mathrm{VGeom}\ar@{=>}@<1ex>[rr]\ar@{=>}@<-1ex>[u]|\times
 &&\mathrm{VSS}\ar@{<=>}[r]\ar@{==>}@<-1ex>[u]_{(A)}\ar@{=>}@<1ex>[ll]|\times
&\mathrm{VPoly}\ar@{==>}@<-1ex>[u]_{(A)}\ar@{=>}@<-1ex>[rr]
 &&\mathrm{DSurf}
} 
\]
where each dashed arrow is a question or conjecture unanswered in this paper,
and each broken arrow with $\times$ in the middle means a false implication.
Question~\ref{que:p and vp}, the \hyperref[conj:tiling]{Tiling Conjecture} and the \hyperref[conj:vtiling]{Virtual Tiling Conjecture} are equivalent to 
(A), (B) and (C), respectively. (D) is equivalent to
the Gromov's conjecture for the groups of the form $D(w)$.
\end{rem}
 
\section{Proof of Theorem~\ref{thm:p but not vg}}\label{sec:nvg}
Let us set $F_3=\langle a,b,c\rangle$ and $F_4=\langle a,b,c,d\rangle$.
Throughout this section, put 
$w_1=bbaaccabc\in F_3$ and
$w_2 = aabbacbccadbdcdd\in F_4$.
A graph $\Gamma$ is \textit{$k$--valent} if the valence of each vertex is $k$.
$\Gamma$ is \textit{$k$--edge-connected} if one cannot disconnect $\Gamma$
by removing the interior of $k-1$ or fewer edges.
In~\cite{Manning:2009p3177}, Manning proved
(a stronger version of) the following theorem.

\begin{theorem}[\cite{Manning:2009p3177}]\label{thm:manning}
Let $w\in F$ be a cyclically reduced word
such that for some $k\ge 3$,
$W(w)$ is a $k$--valent, $k$--edge-connected and non-planar graph.
Then $w$ is not virtually geometric.\qed
\end{theorem}

An example of graphs satisfying the hypothesis in Theorem~\ref{thm:manning}
is the complete bipartite graph $K_{k,k}$ for $k\ge3$.
Manning noted that $W(w_1)$ is $K_{3,3}$,
and hence, $w_1$ is not virtually geometric~\cite{Manning:2009p3177}.
Since $W(w_2)$ is $K_{4,4}$, $w_2$ is not virtually geometric, either. 
On the other hand,
 
\begin{prop}
\begin{enumerate}
\item
$w_1=bbaaccabc$ is polygonal.
\item
$w_2 = aabbacbccadbdcdd$ is polygonal.
\end{enumerate}
\label{prop:w1w2}
\end{prop}

\begin{proof}
(1) 
Let $P_1$ be a polygonal disk whose boundary reads $w_1^2$. Name the edges of $\partial P_1$ as $1,2,\ldots, 18$ so that the edge named by $i$ corresponds to the $i$--th letter of $w_1^2$. Define a side-pairing $\sim_1$ on $P_1$ as shown in Fig.~\ref{fig:w1} (a).
That is to say, the edges of $\partial P_1$ are paired by $\sim_1$ as
\[
\{1,2\},
\{3,7\},
\{4,16\},
\{5,15\},
\{6,18\},
\{8,11\},
\{9,14\},
\{10,17\},
\{12,13\}.
\]

A neighborhood of each vertex in $S_1=P_1/\!\!\sim_1$ is illustrated in Fig.~\ref{fig:w1} (b). One sees that no two incoming edges (or two outgoing edges) of the same label exist at each vertex. Therefore, $S_1$ is a $w_1$--polygonal surface.
Since $\chi(S_1)=|S^{(0)}|-|S^{(1)}|+|S^{(2)}|= 4 - 18/2 + 1= -4 < 1$,
we conclude that $w_1$ is polygonal.

(2)
The proof that $w_2$ is polygonal is almost identical with (1) by considering a side-pairing $\sim_2$ on $P_2$, where $P_2$ is a polygonal disk whose boundary reads $w_2$; see Fig.~\ref{fig:w2} (a).
Specifically, $\sim_2$ identifies the edges of $\partial P_2$ as
\[
\{1,2\},
\{3,12\},
\{4,7\},
\{5,10\},
\{6,14\},
\{8,9\},
\{11,16\},
\{13,15\}.
\]

One again sees that $S_2=P_2/\!\!\sim_2$ is a $w_2$--polygonal surface, from the description of the links in Fig.~\ref{fig:w2} (b). Since $\chi(S_2)=4 - 16/2 + 1= -3 < 1$, $w_2$ is polygonal.
\end{proof}

\begin{rem}\label{rem:prop:w1w2}
(1)
Let $\Gamma$ be a graph immersed in $\cay(F)/F$.
Then $\Gamma$ embeds into $\cay(F)/F'$ for some $[F:F']<\infty$ such that $|\cay(F)/F'^{(0)}|=|\Gamma^{(0)}|$.
In particular, the degree of the cover $\cay(F)/F'\to\cay(F)/F$ can be chosen to be
 $|\Gamma^{(0)}|$; see~\cite{Stallings:1983p596}.
By the proof of Theorem~\ref{thm:polygonal}, we observe that
if $w\in F$ has a closed $w$--polygonal surface $S$,
then $X(w)$ has a finite cover of degree $|S^{(0)}|$
containing a closed surface $S'$ such that $\chi(S') =2 (\chi(S)-m(S))$.
From the proof of Proposition~\ref{prop:w1w2} and Fig.~\ref{fig:w1},
we see that $X(w_1)$ has a finite cover of degree $4$ that contains a closed surface of Euler characteristic $2(\chi(S_1)-1)=-10$. Similarly, a degree--$4$ cover of $X(w_2)$ contains a closed surface of Euler characteristic $2(\chi(S_2)-1)=-8$.

(2)
Let $\Gamma_1=W(w_1^2)$ and $\Gamma_2=W(w_2)$.
$S_1$ has one vertex of valence six and three vertices of valence four.
By the proof of Proposition~\ref{prop:ss} and~\ref{prop:p and ss}, 
$\Gamma_1$ can be written as the union of
one simple cycle of length six
and
three simple cycles of length four.
Also, $\Gamma_2$ is the union of four simple cycles of length four.
\end{rem}

\begin{figure}[htb!]
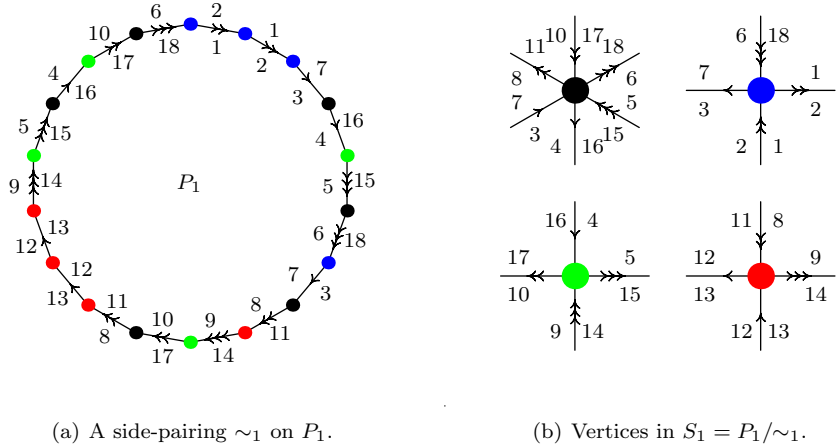

\subfigure[A side-pairing $\sim_1$ on $P_1$.]{
\includegraphics[]{fig/figw1.0}
}
\subfigure[Vertices in $S_1=P_1/\!\!\sim_1$.]{
\includegraphics[]{fig/figw1.1}
}
\caption{Proposition~\ref{prop:w1w2} (1).
Single, double and triple arrows denote $a$-, $b$-and $c$-edges, respectively.
\label{fig:w1}}
\end{figure}

\begin{figure}[htb!]
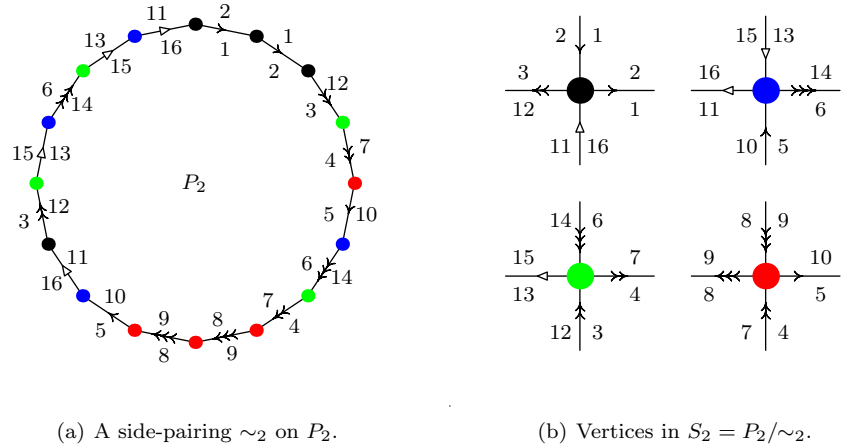

\subfigure[A side-pairing $\sim_2$ on $P_2$.]{
\includegraphics[]{fig/figw2.0}
}
\subfigure[Vertices in $S_2=P_2/\!\!\sim_2$.]{
\includegraphics[]{fig/figw2.1}
}
\caption{Proposition~\ref{prop:w1w2} (2).
Single, double, triple, and white arrows denote
$a$-, $b$-, $c$-and $d$-edges, respectively.
\label{fig:w2}}
\end{figure}

\textbf{Acknowledgement.} 
The author appreciates Cameron Gordon for inspirational conversations.
The author is grateful to Henry Wilton for helpful discussion on Section 2.1, 
and to Alan Reid for his guidance through this work.

\bibliographystyle{plain}
\bibliography{gpfree}

\begin{thebibliography}{10}

\bibitem{Berge:2009p3422}
J~Berge.
\newblock {H}eegaard documentation. {Preprint (c.1990)},
  \url{http://www.math.uic.edu/~t3m}.

\bibitem{Bestvina:2009p3862}
M~Bestvina.
\newblock Questions in geometric group theory.
  \url{http://www.math.utah.edu/∼bestvina}.

\bibitem{Bestvina:1992p456}
M~Bestvina and M~Feighn.
\newblock A combination theorem for negatively curved groups.
\newblock {\em J. Differential Geom.}, 35(1):85--101, 1992.

\bibitem{Calegari:2008p1810}
D~Calegari.
\newblock Surface subgroups from homology.
\newblock {\em Geometry {\&} Topology}, 12(4):1995--2007, 2008.

\bibitem{Canary:1993p4322}
R~D Canary.
\newblock Ends of hyperbolic 3--manifolds.
\newblock {\em J. Amer. Math. Soc.}, 6(1):1--35, 1993.

\bibitem{Gordon:2009p360}
C~McA Gordon and H~Wilton.
\newblock On surface subgroups of doubles of free groups. {Preprint},
  \href{http://arxiv.org/abs/0902.3693v1}{arXiv:0902.3693v1}.

\bibitem{Kim:2009p3867}
S~Kim and H~Wilton.
\newblock On polygonality of words in free groups. {Preprint},
  \href{http://arxiv.org/abs/0910.4709v1}{arXiv:0910.4709v1}.

\bibitem{Manning:2009p3177}
J~F Manning.
\newblock Virtually geometric words and {W}hitehead's algorithm. {Preprint},
  \href{http://arxiv.org/abs/0904.0724v2}{arXiv:0904.0724v2}.

\bibitem{Scott:1978p142}
P~Scott.
\newblock Subgroups of surface groups are almost geometric.
\newblock {\em J. London Math. Soc. (2)}, 17(3):555--565, 1978.

\bibitem{Stallings:1999p3173}
J~R Stallings.
\newblock {W}hitehead graphs on handlebodies. \textit{Geometric group theory
  down under (Canberra, 1996)}, 317--330, de {G}ruyter, {B}erlin, 1999.

\bibitem{Stallings:1983p596}
J~R Stallings.
\newblock Topology of finite graphs.
\newblock {\em Invent. Math.}, 71(3):551--565, 1983.

\bibitem{Stong:1997p4326}
R~Stong.
\newblock Diskbusting elements of the free group.
\newblock {\em Math. Res. Lett.}, 4(2-3):201--210, 1997.

\bibitem{Whitehead:1936p4475}
J~H~C Whitehead.
\newblock On certain sets of elements in a free group.
\newblock {\em Proceedings of the London Mathematical Society}, s 2-41(1):48,
  Jan 1936.

\bibitem{Wilton:2008p7013}
Henry Wilton.
\newblock Hall's theorem for limit groups.
\newblock {\em Geom. Funct. Anal.}, 18(1):271--303, 2008.

\bibitem{Wise:2000p790}
D~T Wise.
\newblock Subgroup separability of graphs of free groups with cyclic edge
  groups.
\newblock {\em Q. J. Math.}, 51(1):107--129, 2000.

\bibitem{Zieschang:1965p4314}
H~Zieschang.
\newblock Simple path systems on full pretzels. \textit{Mat. Sb. (N.S.)}, 66
  (108):230--239, 1965. {T}ranslated in \textit{AMS Translations (2)}, 92:
  127---137, 1970.

\end{thebibliography}

\end{document}